\documentclass[a4paper]{amsart}
\usepackage[centertags]{amsmath}
\usepackage[breaklinks=true,colorlinks]{hyperref}
\usepackage[textwidth=15cm,hcentering]{geometry}
\usepackage{amsfonts}
\usepackage{amssymb}
\usepackage{amsthm}
\usepackage{newlfont}
\usepackage{amscd}
\usepackage{amsmath,amscd}
\usepackage{graphicx}
\usepackage[all]{xy}
\usepackage{verbatim}

\usepackage{color}

\usepackage{eucal}

\makeatletter
\def\@tocline#1#2#3#4#5#6#7{\relax
\ifnum #1>\c@tocdepth 
  \else
    \par \addpenalty\@secpenalty\addvspace{#2}%
\begingroup \hyphenpenalty\@M
    \@ifempty{#4}{%
      \@tempdima\csname r@tocindent\number#1\endcsname\relax
 }{%
   \@tempdima#4\relax
 }%
 \parindent\z@ \leftskip#3\relax \advance\leftskip\@tempdima\relax
 \rightskip\@pnumwidth plus4em \parfillskip-\@pnumwidth
 #5\leavevmode\hskip-\@tempdima #6\nobreak\relax
 \ifnum#1<0\hfill\else\dotfill\fi\hbox to\@pnumwidth{\@tocpagenum{#7}}\par
 \nobreak
 \endgroup
  \fi}
\makeatother

\makeatletter
\ifcsname phantomsection\endcsname
    \newcommand*{\qrr@gobblenexttocentry}[5]{}
\else
    \newcommand*{\qrr@gobblenexttocentry}[4]{}
\fi
\newcommand*{\addsubsection}{%
    \addtocontents{toc}{\protect\qrr@gobblenexttocentry}%
    \subsection}
\makeatother

\usepackage{lipsum}

\newcommand{\adj}{\rightleftarrows}

\newcommand{\QQ}{\mathbb{Q}}
\newcommand{\ZZ}{\mathbb{Z}}

\newcommand{\PP}{\mathbb{P}}

\newcommand{\NN}{\mathbb{N}}
\newcommand{\cA}{\mathcal{A}}

\newcommand{\cC}{\mathcal{C}}
\newcommand{\cD}{\mathcal{D}}

\newcommand{\cG}{\mathcal{G}}

\newcommand{\cS}{\mathcal{S}}
\newcommand{\cT}{\mathcal{T}}

\newtheorem{thm}{Theorem}[subsection]

\newtheorem{cor}[thm]{Corollary}

\newtheorem{ques}[thm]{Question}
\newtheorem{lem}[thm]{Lemma}
\newtheorem{prop}[thm]{Proposition}

\newtheorem{example}[thm]{Example}

\theoremstyle{definition}
\newtheorem{define}[thm]{Definition}

\newtheorem{notat}[thm]{Notation}

\theoremstyle{remark}
\newtheorem{rem}[thm]{Remark}

\DeclareMathOperator{\Ext}{Ext}

\DeclareMathOperator{\im}{Im}

\DeclareMathOperator{\Set}{\cS et}

\DeclareMathOperator{\Grp}{\cG rp}

\DeclareMathOperator{\cf}{cf}

\DeclareMathOperator{\supp}{supp}

\DeclareMathOperator{\ab}{ab}
\DeclareMathOperator{\AbC}{\cA b}

\def\coker{\textrm{coker\,}}

\def\im{\textrm{Im\,}}

\def\lrar{\longrightarrow}

\title{Inverse limits of left adjoint functors on pointed sets}

\author{
Ilan Barnea \and Saharon Shelah}

\address{Department of Mathematics\\
University of Haifa\\
Haifa\\
Israel}
\email{ilanbarnea770@gmail.com}

\address{Department of Mathematics\\
Hebrew University of Jerusalem\\
Jerusalem\\
Israel}
\email{shelah@math.huji.ac.il}

\thanks{The second author was partially supported by European Research Council grant 338821. Publication 1647
.}


\begin{document}

\begin{abstract}
This paper is a continuation of \cite{BaSh}, where we studied the behaviour of the abelianization functor under inverse limits. Our main result in \cite{BaSh} was that if $\mathcal{T}$ is a countable directed poset and $G:\mathcal{T}\longrightarrow\mathcal{G} rp$ is a diagram of groups that satisfies the Mittag-Leffler condition, then the natural map
$$\mathrm{Ab}({\lim}_{t\in\mathcal{T}}G_t)\longrightarrow {\lim}_{t\in\mathcal{T}}\mathrm{Ab}(G_t)$$
is surjective, and its kernel is a cotorsion group. The abelianization is an example of a left adjoint functor from groups to abelian groups.

In this paper we study the behaviour under inverse limits of left adjoint functors from pointed sets to abelian groups. Such functors are classified by abelian groups, where to the abelian group $A$ corresponds the left adjoint functor $L_A:\mathcal{S} \text{et}_*\to\mathcal{A} \text{b}$ given by
$L_A(Y)=\bigoplus_{Y\setminus\{*\}}A.$
If $\mathcal{T}$ is a directed poset and $X:\mathcal{T}\longrightarrow\mathcal{S} \text{et}_*$ is a is diagram of pointed sets, we show that the natural map
$$\rho:L_A({\lim}_{t\in\mathcal{T}}X_t) \longrightarrow{\lim}_{t\in\mathcal{T}}L_A(X_t)$$
is injective. If, in addition, $\mathcal{T}$ is countable and $X$ satisfies the Mittag-Leffler condition, we show that the cokernel of $\rho$ is an algebraically compact group.  Compared with the main result in \cite{BaSh}, algebraically compact is much stronger then cotorsion as it also requires the Ulm length to be $\leq 1$.

We also show that this result, even in its weak form of cotorsion, does not extend to uncountable diagrams. Namely, if $A$ is not the product of a divisible group and a bounded group, we construct a directed poset $\mathcal{T}$ with $|\cT|=2^{\aleph_0}$ and a diagram $X:\mathcal{T}\longrightarrow\mathcal{S} \text{et}_*$, that satisfies the Mittag-Leffler condition, such that the cokernel of $\rho$ is not cotorsion.

If $\mathcal{T}$ is any directed poset and $X:\mathcal{T}\longrightarrow\mathcal{S} \text{et}^f_*$
is a diagram of finite pointed sets, we show that if $A$ is cotorsion then the cokernel of $\rho$ is cotorsion. We construct a counterexample showing that this last result is best possible in terms of $A$ for uncountable diagrams. We also prove various other results on the cokernel of $\rho$ depending on properties of the abelian group $A$ and the diagram $X$.
\end{abstract}

\maketitle
\tableofcontents

\section{Introduction}
In \cite{BaSh}, we investigated the behaviour of the abelianization functor with respect to inverse limits. We showed that if ${\cT}$ is a countable directed poset and $G:{\cT}\longrightarrow\Grp$ is a diagram of groups that satisfies the Mittag-Leffler condition, then the natural map
$$({\lim}_{\mathcal{T}}G)_{\ab}\longrightarrow {\lim}_{\mathcal{T}}(G_{\ab})$$
is surjective, and its kernel is a cotorsion group (see after Theorem \ref{t:main alg com} for the definitions of the Mittag-Leffler condition and cotorsion groups).

As mentioned in \cite{BaSh}, the abelianization is an example of a left adjoint functor. Namely, it is the left adjoint to the inclusion functor from abelian groups to groups. If $\cC$ is any category and $F:\cC\to\AbC$ is a left adjoint functor, then $F$ commutes with all direct limits that exist in $\cC$. Generalizing the philosophy behind \cite{BaSh}, it is natural to ask about the behaviour of left adjoint functors with respect to inverse limits.

In this paper, we address this question when $\cC=\Set_*$ is the category of pointed sets. This is a rather simple starting point, and indeed we are able to show much more then we did for the abelianization functor in \cite{BaSh}. Note that a functor $\Set_*\to\AbC$ is a left adjoint iff it commutes with all direct limits (see for instance \cite{AR}). A recent result by Hartl and Vespa \cite[Theorem 3.12]{HaVe} completely classifies such functors in terms of abelian groups. Specifically, they show that assigning to the abelian group $A$ the left adjoint functor $L_A:\Set_*\to\AbC$ given by
$$L_A(Y)=Y\wedge A=\bigoplus_{Y\setminus\{*\}}A$$
($*\in Y$ being the special point), can be extended to an equivalence of categories.


In this paper we investigate the behaviour of such functors $L_A$ with respect to inverse limits. Specifically, let $A$ be an abelian group, $\cD$ a small category and $X:\cD\lrar \Set_*$ a diagram of pointed sets. There is always a natural map
$$\rho:L_A({\lim}_\cD X)\lrar{\lim}_\cD L_A(X),$$
and we investigate how far is this map from being an isomorphism, by studying its kernel and cokernel. We will especially be interested whether the kernel and cokernel are cotorsion, extending our result in \cite{BaSh}, as explained above.
Our results depend on properties of the abelian group $A$ and the diagram $X$.

We will be using the following definition: An abelian group is called \emph{almost (uniquely) divisible} if it is the direct sum of a (uniquely) divisible group and a bounded group. Note that an almost divisible group is algebraically compact (see Corollary \ref{c:AD is AC}), and thus cotorsion.

The only result we show in this paper for a general diagram category $\cD$ is the following:
\begin{thm}[see Theorem \ref{t:p divisible gen}]\label{t:p divisible gen0}
  We have the following:
\begin{enumerate}
  \item If $A$ is $m$-bounded then the kernel and cokernel of $\rho$ are $m$-bounded.
  \item If $A$ is uniquely divisible then the kernel and cokernel of $\rho$ are uniquely divisible.
  \item If $A$ is almost uniquely divisible then the kernel and cokernel of $\rho$ are almost uniquely divisible (and thus cotorsion).
\end{enumerate}
\end{thm}

We will concentrate in this paper mainly on the case when the diagram category is a directed poset (considered as a category which has a single morphism $t\to s$ whenever $t\geq s$). So from now on we assume $\cD=\cT$ is a directed poset. Our first observation is the following:
\begin{thm}[Lemma \ref{l:inj}]\label{t:inj}
The map $\rho$ is injective.
\end{thm}
We thus only need to study the cokernel of $\rho$. Another general result for a directed poset we show is:
\begin{thm}[Corollary \ref{c:tor free}]
If $A$ is torsion free then the cokernel of $\rho$ is torsion free.
\end{thm}

First suppose that $\cT$ is countable. Since $L_A$ sends epimorphisms to epimorphisms it follows easily from \cite[Theorem 0.0.4]{BaSh} or \cite[Statement 1]{Akh} that if $X$ satisfies the Mittag-Leffler condition, the cokernel of $\rho$ is cotorsion. However, in this paper we prove a much stronger result.
\begin{thm}[Theorem~\ref{t:alg compact}]\label{t:main alg com}
If $\cT$ is countable and $X$ satisfies the Mittag-Leffler condition, then the cokernel of $\rho$ is algebraically compact (and thus cotorsion).
\end{thm}

Recall that $X$ satisfies the Mittag-Leffler condition if for every $t\in\cT$ there exists $s\geq t$ such that for every $r\geq s$ we have $$\im (X(s)\to X(t))=\im (X(r)\to X(t)).$$
If $X$ has surjective connecting homomorphisms or $X$ is diagram of finite pointed sets then $X$ satisfies Mittag-Leffler condition.

Recall also that an abelian group $A$ is called \emph{cotorsion} if it satisfies $\Ext(\QQ,A)=0$ (or, equivalently, $\Ext(F,A)=0$ for any torsion free abelian group). The group $A$ is called \emph{algebraically compact} if it is cotorsion, and its Ulm length does not exceed $1$. Algebraically compact groups can be completely classified by a countable collection of cardinals (see Balcerzyk \cite{Bal}).

\begin{rem}
  The question of describing the cardinal invariants corresponding to $\coker(\rho)$ in Theorem \ref{t:main alg com}, in terms of cardinal invariants of the diagram $X$ and the abelian group $A$, will be addressed in a future paper. This will also reveal which algebraically compact groups can appear as $\coker(\rho)$.
\end{rem}

\begin{rem}
  In Example \ref{e:finite} we construct a countable directed poset $\cT$, together with a diagram $X:\cT\lrar\Set_*$, composed of finite sets (and thus satisfies Mittag-Leffler condition), such that
  $$\coker (\rho)\cong \prod_{\NN}A/\bigoplus_{\NN}A.$$
  Thus, Theorem \ref{t:main alg com} can be viewed as an extension of an old result of Hulanicki \cite{Hul}. This also gives many examples where $\coker(\rho)$ is non trivial.
\end{rem}

Another result we show for countable diagrams is

\begin{thm}[see Corollary \ref{c:almost divisible}]\label{c:almost divisible0}
If $\cT$ is countable and $X$ is satisfies the Mittag-Leffler condition, we have the following:
\begin{enumerate}
  \item If $A$ is $p$-divisible then the cokernel of $\rho$ is $p$-divisible.
  \item If $A$ is divisible then the cokernel of $\rho$ is divisible.
  \item If $A$ is almost divisible then the cokernel of $\rho$ is almost divisible.
\end{enumerate}
\end{thm}

We now turn to the case of uncountable $\cT$. Recall that the poset $\cT$ is called \emph{$\aleph_1$-directed} if for every countable subset $S\subseteq\cT$ there exists $t\in \cT$ such that $t\geq s$ for every $s\in S$. We show
\begin{thm}[Corollary \ref{c:a1 directed}]\label{t:main a1 directed}
If $\cT$ is $\aleph_1$-directed then $\rho$
is an isomorphism.
\end{thm}

Combining Theorems \ref{t:main alg com} and \ref{t:main a1 directed} it is not hard to show
\begin{thm}[Corollary \ref{c:alg compact}]\label{t:main ord alg com}
If $\cT=\lambda$ is an ordinal and $X$ satisfies the Mittag-Leffler condition, then the cokernel of $\rho$
is algebraically compact (and thus cotorsion).
\end{thm}

Given Theorems \ref{t:main alg com} and \ref{t:main ord alg com}, it is natural to ask whether they remain true for any directed poset $\cT$.
 The answer depends on $A$; If $A$ is almost uniquely divisible, this follows from Theorem \ref{t:p divisible gen0}, regardless of the Mittag-Leffler condition. We also show:

\begin{thm}[Theorem \ref{t:counter}]
If $A$ is not almost divisible, then there exists a diagram $X$, with $|\cT|=2^{\aleph_0}$, that satisfies the Mittag-Leffler condition, such that the cokernel of $\rho$ is not cotorsion.
\end{thm}

Almost divisible groups are groups of the form
$$(\bigoplus_{\alpha<\lambda_0} \QQ )\oplus\bigoplus_{p\in \PP}\bigoplus_{\alpha<\lambda_p}\ZZ(p^\infty)\oplus B,$$
with $(\lambda_p)_{p\in\PP\cup \{0\}}$ cardinals and $B$ bounded, and almost uniquely divisible groups are such groups with $\lambda_p=0$ for every $p\in\PP$.
Thus, the following question remains, which we still didn't answer:
\begin{ques}\label{q:Zpinf}
  Suppose $A=\ZZ(p^\infty)$ for some prime $p$. Is the cokernel of $\rho$ cotorsion for every diagram $X$ that satisfies the Mittag-Leffler condition?
\end{ques}

Our results show that the Mittag-Leffler condition loses its relevance for our purposes when $|\cT|\geq 2^{\aleph_0}$, in the sense that it doesn't allow us to prove a stronger result than we would otherwise have (see still Question \ref{q:Zpinf}). In order to prove a stronger result we need to impose a stronger condition on uncountable diagrams $X$ than the Mittag-Leffler condition. Namely, we assume that $X$ is composed of finite sets. We show

\begin{thm}[see Theorems \ref{t:p divisible}]\label{t:cotorsion0}
If $X$ is composed of finite sets, we have the following:
\begin{enumerate}
  \item If $A$ is $p$-divisible then the cokernel of $\rho$ is $p$-divisible.
  \item If $A$ is divisible then the cokernel of $\rho$ is divisible.
\item If $A$ is cotorsion then the cokernel of $\rho$ is cotorsion.
\end{enumerate}
\end{thm}

Part (3) of Theorem \ref{t:cotorsion0} is best possible in the sense that we show

\begin{thm}[see Theorem \ref{t:finite counter}]\label{t:finite counter0}
   If $A$ is not cotorsion, then for every cardinal $\lambda>\aleph_0$ there exists a diagram $X$, with $|\cT|=\lambda$, composed of finite sets, such that the cokernel of $\rho$ is not cotorsion.
\end{thm}

\begin{rem}
  We have an adjoint pair
  $$(-)_+:\Set\adj\Set_*:U,$$
  where $U$ is the forgetful functor and $(-)_+$ adds a disjoint basepoint. Note that $(-)_+$ is faithful and essentially surjective. For every $Y\in\Set$ we have a natural isomorphism
  $$L_A(Y_+)=Y_+\wedge A\cong Y\otimes A.$$
  Thus, if $Z:\cT\to\Set$ is a diagram of sets, the natural map
  $$({\lim}_\cT Z)\otimes A\lrar{\lim}_\cT (Z\otimes A)$$
  is naturally isomorphic to the natural map
  $$L_A({\lim}_\cT (Z_+))\cong L_A(({\lim}_\cT Z)_+)\lrar{\lim}_\cT L_A(Z_+).$$
  It follows that all the theorems above remain true if we replace $\Set_*$ by $\Set$ and $L_A$ by a functor of the form $(-)\otimes A$.
\end{rem}

\begin{rem}
  A different kind of extension of our results in \cite{BaSh} appears in \cite{Akh}. Note that the abelianization functor is just the first homology functor (with integer coefficients) on the category of groups. Akhtiamov studies the behaviour of higher homology functors on the category of groups with respect to inverse limits.
\end{rem}

\section{Almost divisible groups}
The main purpose of this section is to introduce the notion of an \emph{almost divisible group} and to study some equivalent definitions of it. Throughout this section, we let $A$ be an abelian group.

Recall that the abelian group $A$ is called (uniquely) divisible, if for every $a\in A$ and $n\in\NN$ there exists (a unique) $b\in A$ such that $a=nb$. If $p$ is a prime then $A$ is called (uniquely) $p$-divisible, if for every $a\in A$ there exists (a unique) $b\in A$ such that $a=pb$.

Recall from \cite[page 154]{Fu1} that if $p$ is a prime, the $p$-length of $A$, denoted $l_p(A)$, is defined to be the smallest ordinal $\lambda$ for which $p^\lambda A$ is $p$-divisible. Note that $A$ is divisible iff for every prime $p$ we have $l_p(A)=0$. This motivates the following:

\begin{define}
   The group $A$ is called \emph{almost divisible} if for almost every prime $p$ we have $l_p(A)=0$ and for every prime $p$ we have $l_p(A)<\omega$.
\end{define}
\begin{lem}\label{l:lp}
  If $p\neq q$ are primes and $l_p(A)<\omega$ then $l_p(A)=l_p(qA)$.
\end{lem}

\begin{proof}
  We have $p^{l_p(A)}A=p^{l_p(A)+1}A$ so we have $$p^{l_p(A)}(qA)=qp^{l_p(A)}A=qp^{l_p(A)+1}A=p^{l_p(A)+1}(qA),$$ which shows that $l_p(A)\geq l_p(qA)$.

  We are left to show that $l_p(A)\leq l_p(qA)$ or $p^{l_p(qA)}A\subseteq p^{l_p(qA)+1}A$. So let $a\in p^{l_p(qA)}A$. There exists $c\in A$ such that $a=p^{l_p(qA)}c$. Since $qc\in qA$ we have
  $$qa\in p^{l_p(qA)}(qA)=p^{l_p(qA)+1}(qA)\subseteq p^{l_p(qA)+1}A.$$
  Thus, there exists $b\in p^{l_p(qA)}A$ such that $qa=pb$. Since $p\neq q$ are primes, there exist $m,n\in\ZZ$ such that $mp+nq=1$. It follows that
  $$a=mpa+nqa=mpa+npb=p(ma+nb).$$
  But $a,b\in p^{l_p(qA)}A$ so $d:=ma+nb\in p^{l_p(qA)}A$ and $a=pd\in p^{l_p(qA)+1}A$.
\end{proof}

\begin{notat}
  If $q=(q_n)_{n\in\NN}$ is a sequence in $\ZZ$ we define for every $n\in\NN$
   $$q_{<n}:=\prod_{l<n}q_l.$$
\end{notat}

\begin{prop}\label{p:criterion}
   The group $A$ is almost divisible iff for every sequence $(q_n)_{n\in\NN}$ in $\ZZ$ there exists $m\in\NN$ such that for every $n>m$ we have
  $q_{<n} A=q_{<m} A.$
\end{prop}

\begin{proof}
  First suppose that the condition of the proposition holds. Let $p$ be a prime number. Define, for every $n\in\NN$, $q_n:=p$. By the condition of the proposition, there exists $m\in\NN$ such that for every $n>m$ we have
  $p^n A=p^{m}A.$ Thus $l_p(A)\leq m<\omega$.

  Now suppose (to derive a contradiction) that there exist infinitely many different primes $q_0,q_1,\dots$ such that for every $i\in\NN$ we have $l_{q_i}(A)>0$. By the condition of the proposition, there exists $m\in\NN$ such that for every $n>m$ we have
  $q_{<n} A=q_{<m} A.$
  In particular, we have
  $$q_mq_{<m} A=q_{<m+1} A=q_{<m} A,$$
  so $l_{q_m}(q_{<m} A)=0.$
  But according to Lemma \ref{l:lp} we have
  $$l_{q_m}(q_{<m} A)=l_{q_m}(A)>0,$$
  which is a contradiction. Thus $A$ is almost divisible.

  Now suppose that $A$ is almost divisible. Then there exist $n\in\NN$ and different primes $p_1,\dots,p_n$ such that $l_{p_i}(A)>0$ for $i=1,\dots,n$ and $l_{q}(A)=0$ for every prime $q\notin \{p_1,\dots,p_n\}$.

  Let $(k_m)_{m\in\NN}$ be a sequence in $\ZZ$. We need to show that there exists $m\in\NN$ such that for every $n>m$ we have
  $k_{<n} A=k_{<m} A.$
   If $k_m=0$ for some $m\in\NN$ this is trivial, so we can assume that $(k_m)_{m\in\NN}$ is a sequence in $\ZZ\setminus\{0\}$. We can thus choose recursively increasing sequences $(N_m)_{m\in\NN}$, $(\mu_{m,i})_{m\in\NN}$ ($i=1,\dots,n$) in $\NN$ and a sequence $(q_m)_{m\in\NN}$ in $\PP\setminus\{p_1,\dots,p_n\}$ such that for every $m\in\NN$ we have
  $$k_{<m}=\prod_{i=1}^n p_i^{\mu_{m,i}}q_1\cdots q_{N_m}.$$
  For every $q\in\PP\setminus\{p_1,\dots,p_n\}$ we have $l_q(A)=0$ so for every $m\in\NN$ we have
  $$k_{<m} A=(\prod_{i=1}^n p_i^{\mu_{m,i}})A.$$

  Let $i=1,\dots,n$. If $(\mu_{m,i})_{m\in\NN}$ is unbounded we choose $m_i\in\NN$ such that $\mu_{m_i,i}\geq l_{p_i}(A)$, otherwise we choose $m_i\in\NN$ such that
  $\mu_{m_i,i}=\max\{\mu_{m,i}:m\in\NN\}.$
  We now define $m:=\max\{m_1,\dots,m_n\}$.
  Let $t>m$.
\begin{lem}
     For every $0\leq j\leq n$ we have
  $$(\prod_{i=1}^j p_i^{\mu_{t,i}})A=(\prod_{i=1}^j p_i^{\mu_{m,i}})A.$$
\end{lem}

\begin{proof}
We prove the lemma by induction on $j$. When $j=0$ the lemma is clear. Now suppose we have proven the lemma for some $j<n$, and let us prove it for $j+1$. By Lemma \ref{l:lp}, we have
$$l_{p_{j+1}}((\prod_{i=1}^j p_i^{\mu_{t,i}})A)= l_{p_{j+1}}(A).$$

Suppose $l_{p_{j+1}}(A)\leq \mu_{m_{j+1},j+1}$. Then, since $\mu_{(-),j+1}$ is monotone increasing and using the induction hypothesis, we obtain
$$p_{j+1}^{\mu_{t,j+1}}(\prod_{i=1}^j p_i^{\mu_{t,i}})A= p_{j+1}^{\mu_{m,j+1}}(\prod_{i=1}^j p_i^{\mu_{t,i}})A=
p_{j+1}^{\mu_{m,j+1}}(\prod_{i=1}^j p_i^{\mu_{m,i}})A.$$
Otherwise, we have $\mu_{m,j+1}\leq\mu_{t,j+1}\leq \mu_{m_{j+1},j+1}\leq\mu_{m,j+1}$, so $\mu_{t,j+1}=\mu_{m,j+1}$ and the lemma follows easily using the induction hypothesis.
\end{proof}
Taking $j=n$ in the lemma above we obtain
  $$k_{<t} A=(\prod_{i=1}^n p_i^{\mu_{t,i}})A=(\prod_{i=1}^n p_i^{\mu_{m,i}})A=k_{<m} A,$$
which finishes our proof.
\end{proof}

\begin{prop}
   The group $A$ is almost divisible and reduced iff $A$ is bounded.
\end{prop}

\begin{proof}
  Suppose that $A$ is almost divisible and reduced. Then there exist $n\in\NN$ and different primes $p_1,\dots,p_n$ such that $l_{p_i}(A)>0$ for $i=1,\dots,n$ and $l_{q}(A)=0$ for every prime $q\notin \{p_1,\dots,p_n\}$. We define
  $$m:=\prod_{i=1}^{n}p_i^{l_{p_i}(A)}\in\NN\setminus\{0\}.$$

  Using Lemma \ref{l:lp}, we see that for every $i=1,\dots,n$ we have
$$l_{p_{i}}(mA)=l_{p_{i}}(p_i^{l_{p_i}(A)}A)=0$$
and for every $q\in\PP\setminus\{p_1,\dots,p_n\}$ we have
$$l_q(mA)=l_q(A)=0.$$
Thus, $mA$ is divisible. But $mA$ is a subgroup of $A$ and $A$ is reduced so $mA=\{0\}$, and $A$ is bounded.

Now suppose that $A$ is bounded, so there exists $m\in\NN\setminus\{0\}$ such that $mA=0$. By decomposing $A$ into its divisible and reduced parts, it is not hard to see that $A$ is reduced.

Let us write
$m=\prod_{i=1}^{n}p_i^{l_{p_i}},$
where $p_1,\dots,p_n$ are different primes.
Using Lemma \ref{l:lp}, we see that for every $i=1,\dots,n$ we have $l_{p_i}(p_i^{l_{p_i}}A)=l_{p_i}(mA)=0$ so $l_{p_i}(A)\leq l_{p_i}$
and for every $q\in\PP\setminus\{p_1,\dots,p_n\}$ we have
$$0=l_q(mA)=l_q(A).$$
Thus, $A$ is almost divisible.
\end{proof}

By decomposing an abelian group into its divisible and reduced parts we obtain the following:

\begin{cor}\label{c:AD is AC}
An abelian group is almost divisible iff it is the direct sum of a divisible group and a bounded group. Using \cite[Theorems 21.2 and 27.5]{Fu1} we see that an almost divisible abelian group is algebraically compact.
\end{cor}

This corollary naturally leads to the following:
\begin{define}
  An abelian group is called \emph{almost uniquely divisible} if it is the direct sum of a uniquely divisible group and a bounded group. In particular, an almost uniquely divisible group is almost divisible.
\end{define}

\section{A general diagram category}
Throughout this section, we let $A$ be an abelian group, $\cD$ a small category and $X:\cD\lrar \Set_*$ a diagram of pointed sets. The purpose of this paper is to study the kernel and cokernel of the natural map
$$\rho:({\lim}_\cD X)\wedge A\lrar{\lim}_\cD(X\wedge A).$$

Let $Y$ be a pointed set. Recall that
$$Y\wedge A=\bigoplus_{Y\setminus\{*\}}A.$$
Elements $y\in Y\wedge A$ can be identified with pointed maps $s\mapsto y_s:Y\to A$ such that $\supp(y):=\{s\in Y:y_s\neq 0\}$ is finite. For every $s\in Y\setminus\{*\}$ let $i_s:A\to Y\wedge A$ be the inclusion in the $s$-coordinate and let $i_*:A\to Y\wedge A$ be the zero map. If $s\in Y$ and $v\in A$ we denote $sv:=i_s(v)$. Then clearly we have
$$y=\bigoplus_{s\in Y\setminus\{*\}}y_s=\sum_{s\in \supp(y)}sy_s=\sum_{s\in Y}sy_s.$$
We denote $|y|:=|\supp(y)|$.
Note, that if $h\in {\lim}_\cD (X\wedge A)$ and $s\to t$ is a morphism in $\cD$, then $X(s\to t)(h(s))=h(t)$ so it is easy to see that $|h(s)|\geq |h(t)|$.

Let us describe the natural map $\rho$ more explicitly. Since every element in $({\lim}_\cD X)\wedge A$ is a finite sum of element of the form $xv$ for $x\in {\lim}_\cD X$ and $v\in A$, it is enough to define $\rho$ on these elements. This is given by
$$\rho(xv)(s)=x(s)v\in X(s)\wedge A,$$
for every $s\in\cD$.

\begin{lem}\label{l:oplus}
  If $A'$ is another abelian group, with natural map $\rho'$ (relative to $X$), then we have natural isomorphisms
$$({\lim}_\cD X)\wedge (A\oplus A')\cong (({\lim}_\cD X)\wedge A)\oplus (({\lim}_\cD X)\wedge A')$$
and
$$\:\:{\lim}_\cD (X\wedge (A\oplus A'))\cong ({\lim}_\cD (X\wedge A))\oplus ({\lim}_\cD (X\wedge A')),$$
under which, the natural map for $A\oplus A'$ becomes $\rho\oplus\rho'$. In particular, the (co)kernel of the natural map for $A\oplus A'$ is direct sum of the (co)kernels of $\rho$ and $\rho'$.
\end{lem}

\begin{proof}
Since in $\AbC$ a finite direct sum is also a finite direct product, we have
  $$\:\:{\lim}_\cD (X\wedge (A\oplus A'))\cong {\lim}_\cD ((X\wedge A)\oplus (X\wedge A'))\cong ({\lim}_\cD (X\wedge A))\oplus ({\lim}_\cD (X\wedge A')),$$
  and the result is easily seen.
\end{proof}

Recall that the abelian group $A$ is called $m$-bounded ($m\in\NN$) if $mA=\{0\}$.

\begin{thm}\label{t:p divisible gen}
We have the following:
\begin{enumerate}
  \item If $A$ is $m$-bounded then the kernel and cokernel of $\rho$ are $m$-bounded.
  \item If $A$ is uniquely divisible then the kernel and cokernel of $\rho$ are uniquely divisible.
  \item If $A$ is almost uniquely divisible then the kernel and cokernel of $\rho$ are almost uniquely divisible.
\end{enumerate}
\end{thm}

\begin{proof}
  \begin{enumerate}
    \item It is not hard to see that $A$ is $m$-bounded iff there exists a structure of a $\ZZ(m)$-module on $A$. (If such a structure exists, it is necessarily unique.) So suppose that $A$ is $m$-bounded. Then $A$ is a $\ZZ(m)$-module and the functor $(-)\wedge A$ on pointed sets takes values in the category of $\ZZ(m)$-modules. Thus, the natural map $\rho$ is a map between $\ZZ(m)$-modules and its kernel and cokernel are also $\ZZ(m)$-modules.
    \item Similar to (1), with $\ZZ(m)$ replaced by $\QQ$, using the fact that $A$ is uniquely divisible iff there exists a structure of a $\QQ$-vector space on $A$ (and if such a structure exists, it is necessarily unique).
    \item Suppose that $A$ is almost uniquely divisible. There exists a direct sum decomposition $A=Q\oplus B$ with $Q$ uniquely divisible and $B$ bounded. By Lemma \ref{l:oplus}, the (co)kernel of $\rho$ is the direct sum of the (co)kernels of the natural maps for $Q$ and $B$. Now the result follows from (1) and (2).
  \end{enumerate}
\end{proof}

\section{A general directed poset}
From now until the end of this paper we assume that our diagram category is a directed poset (considered as a category which has a single morphism $t\to s$ whenever $t\geq s$). So, throughout this section, we let $A$ be an abelian group, $\cT$ a directed poset and $X:\cT\lrar \Set_*$ a diagram of pointed sets. Recall that we have a natural map
$$\rho:({\lim}_\cT X)\wedge A\lrar{\lim}_\cT(X\wedge A).$$

\begin{lem}\label{l:inj}
  The map $\rho$ is injective.
\end{lem}

\begin{proof}
  We need to show that $\ker\rho=\{0\}$. So assume (to derive a contradiction)
  $m\geq 1$, $\sum_{k=1}^{m}x_kv_k\in \ker\rho$, $x_1,\dots,x_m\in \lim_\cT X\setminus\{*\}$ are different and $v_1,\dots,v_m\in A\setminus\{0\}$.

  Let $x_{0}:=*\in \lim_\cT X$ be the special point and let $k,l\in\{0,\dots,m\}$. Since $x_k\neq x_l$, there exists $t_{k,l}\in\cT$ such that $x_k(t_{k,l})\neq x_l(t_{k,l})$. Since $x_k,x_l\in \lim_\cT X$, we also have $x_k(s)\neq x_l(s)$ for every $s\geq t_{k,l}$. But $\cT$ is directed, so we can find $t\in\cT$ such that $t\geq t_{k,l}$ for every $k,l\in\{0,\dots,m\}$. Thus, we obtain that $*,x_1(t),\dots,x_m(t)\in X(t)$ are all different and we have
  $$\rho(\sum_{k=1}^{m}x_kv_k)(t)=\sum_{k=1}^{m}\rho(x_kv_k)(t)= \sum_{k=1}^{m}x_k(t)v_k\neq 0,$$
  contradicting the fact that $\sum_{k=1}^{m}x_kv_k\in \ker\rho$.
\end{proof}

Lemma \ref{l:inj} tells us that $\ker(\rho)=\{0\}$, so we only need to understand $\coker(\rho)$. From now on we will identify $({\lim}_\cT X)\wedge A$ with $\im(\rho)$ using $\rho$.
We define functors:
$$G(X,-):={\lim}_\cT (X\wedge (-)):\AbC\to\AbC,$$
$$K(X,-):=({\lim}_\cT X)\wedge (-):\AbC\to\AbC.$$
The natural map $\rho$ is a subfunctor inclusion $K(X,-)\hookrightarrow G(X,-)$, and we define
$$H(X,-):=\coker(\rho)=G(X,-)/K(X,-):\AbC\to\AbC.$$
For $h\in G(X,A)$, we will denote by $[h]$ the corresponding element in $H(X,A)$.

\begin{prop}\label{p:bounded}
  Let $h\in G(X,A)$. Then $h\in K(X,A)$ iff $\{|h(t)|:t\in\cT\}$ is bounded.
\end{prop}

\begin{proof}
  Suppose that $h\in K(X,A)=({\lim}_\cT X)\wedge A$. Then $h=\sum_{k=1}^{m}x_kv_k$ for some $x_i\in \lim_\cT X$ and $v_i\in A$. Thus, for every $t\in\cT$ we have
  $$h(t)=\sum_{k=1}^{m}(x_kv_k)(t)=\sum_{k=1}^{m}x_k(t)v_k,$$
  which shows that $|h(t)|\leq m$.

  Now suppose that $m:=\sup\{|h(t)|:t\in\cT\}\in\NN.$
  Choose $t_0\in \cT$ such that $m=|h(t_0)|$. We can find different $x_{1,t_0},\dots,x_{m,t_0}\in X(t_0)\setminus\{*\}$ and $v_1,\dots,v_m\in A\setminus\{0\}$ such that $h(t_0)=\sum_{k=1}^{m}x_{k,t_0}v_k$.

  Let $t\geq t_0$. Since $|h(-)|$ is monotone increasing we have that $m=|h(t)|$. Thus, there are unique (different) $x_{1,t},\dots,x_{m,t}\in X(t)$ such that $\phi_{t,t_0}(x_{k,t})=x_{k,t_0}$ and $h(t)=\sum_{k=1}^{m}x_{k,t}v_k$. It is not hard to verify that if $s\geq t\geq t_0$ then $\phi_{s,t}({x_{k,s}})=x_{k,t}$.

  Now let $k=1,\dots,m$ and let $t\in\cT$. Since $\cT$ is directed, we can find $s\in\cT$, such that $s\geq t$ and $s\geq t_0$. We define $x_k(t):=\phi_{s,t}(x_{k,s})\in X(t)$. It can be verified that this definition is independent of the choice of $s$ and defines an element $x_k\in \lim_\cT X$.

  We claim that $h=\sum_{k=1}^{m}x_{k}v_k$.
  So let $t\in\cT$ and choose $s\in\cT$ such that $s\geq t,t_0$. By the construction above we have that $x_k(t)=\phi_{s,t}(x_{k,s})$ and $h(s)=\sum_{k=1}^{m}x_{k,s}v_k$, so that
  $$h(t)=\phi_{s,t}(h(s))=\phi_{s,t}(\sum_{k=1}^{m}x_{k,s}v_k)= \sum_{k=1}^{m}\phi_{s,t}(x_{k,s}v_k)=$$
  $$ \sum_{k=1}^{m}\phi_{s,t}(x_{k,s})v_k=
  \sum_{k=1}^{m}x_{k}(t)v_k=\sum_{k=1}^{m}(x_{k}v_k)(t)= (\sum_{k=1}^{m}x_{k}v_k)(t).$$
\end{proof}

\begin{cor}\label{c:tor free}
  If $A$ is torsion free then $H(X,A)$ is torsion free.
\end{cor}

\begin{proof}
  Let $k>0$ and let $h\in G(X,A)$. Suppose that $kh\in K(X,A)$. We need to show that $h\in K(X,A)$. Since $A$ is torsion free, for every $t\in\cT$ we have $|kh(t)|=|h(t)|$. Thus
  $$\{|h(t)|:t\in\cT\}=\{|kh(t)|:t\in\cT\},$$
  and the result follows from Proposition \ref{p:bounded}.
\end{proof}

\begin{cor}\label{c:a1 directed}
  If $\cT$ is $\aleph_1$-directed then $H(X,A)=0$.
\end{cor}

\begin{proof}
  Let $h\in G(X,A)$. By Proposition \ref{p:bounded} we need to show that $\{|h(t)|:t\in\cT\}$ is bounded. Suppose that $\{|h(t)|:t\in\cT\}$ is unbounded. For every $n\in\NN$ we can choose $t_n\in\cT$ such that $|h(t_n)|\geq n$. Since $\cT$ is $\aleph_1$-directed, there exists $s\in\cT$ such that $s\geq t_n$ for every $n\in\NN$. Now for every $n\in\NN$ we have $|h(s)|\geq |h(t_n)|\geq n$, which is a contradiction.
\end{proof}

We now look on the case when the diagram is composed of finite sets.
\begin{thm}\label{t:p divisible}
Suppose $X:\cT\lrar \Set^f_*$ is a diagram of finite pointed sets. Then we have the following:
\begin{enumerate}
  \item If $A$ is $p$-divisible then $G(X,A)$ and $H(X,A)$ are $p$-divisible.
  \item If $A$ is divisible then $G(X,A)$ and $H(X,A)$ are divisible.
  \item If $A$ is cotorsion then $G(X,A)$ and $H(X,A)$ are cotorsion.
\end{enumerate}
\end{thm}


\begin{proof}\
\begin{enumerate}
  \item The result for $H$ follows easily from the result for $G$, so it is enough to show that $G(X,A)$ is $p$-divisible.
So let $f\in G(X,A)$ and we need to show that there exists $v\in G(X,A)$ such that $pv=f$.

Let $s\in\cT$. The abelian group $X(s)\wedge A$ is $p$-divisible, so there exists $g(s)\in X(s)\wedge A$ such that $pg(s)=f(s).$
We claim that it is enough to show that there exists $$h\in\prod_{s\in\cT}(X(s)\wedge A)[p]$$
such that for every $s<t$ in $\cT$ we have
$$h(s)=\phi_{t,s}(h(t))+\phi_{t,s}(g(t))-g(s).$$
Because, given such $h$ we can define for every $s\in\cT$
$$v(s):=g(s)+h(s)\in X(s)\wedge A.$$
Now, for every $s<t$ in $\cT$ we have
$$\phi_{t,s}(v(t))=\phi_{t,s}(g(t))+\phi_{t,s}(h(t))= g(s)+h(s)=v(s),$$
so $v\in G(X,A)$. Furthermore, for every $s\in\cT$ we have
$$pv(s)=pg(s)+ph(s)=f(s)+0=f(s),$$
so $pv=f$ in $G(X,A)$.

For every $s\in\cT$, since $X(s)$ is finite, we have that
$$(X(s)\wedge A)[p]= A[p]^{X(s)\setminus\{*\}}.$$
For $s<t$ in $\cT$ the structure map
$\phi_{t,s}:(X(t)\wedge A)[p]\to (X(s)\wedge A)[p]$
is given by a
$$(X(s)\setminus\{*\})\times (X(t)\setminus\{*\})$$
matrix, with integer coefficients, which we denote $\Phi_{t,s}$. In fact, for $a\in X(s)\setminus\{*\}$ and $b\in X(t)\setminus\{*\}$, we have
$$(\Phi_{t,s})_{a,b}=
\begin{cases}
   1 & \text{ if $a=\phi_{t,s}(b)$,}\\
   0 & \text{ else.}
\end{cases}$$
Thus, it is enough to show that there exists $$h\in\prod_{s\in\cT}A[p]^{X(s)\setminus\{*\}}= \prod_{s\in\cT}\prod_{X(s)\setminus\{*\}}A[p]$$
such that for every $s<t$ in $\cT$ and $a\in X(s)\setminus\{*\}$ we have
$$h(s)_a=\sum_{b\in X(t)\setminus\{*\}}(\Phi_{t,s})_{a,b}h(t)_b+(\phi_{t,s}(g(t))_a -g(s)_a).$$
Note that for every such equation we have
$$\phi_{t,s}(g(t))_a -g(s)_a\in A[p],$$
so we obtain a system of equations over the abelian group $A[p]$ as defined in \cite[Section 22]{Fu1}. Specifically, we have an equation for every $s<t$ in $\cT$ and $a\in X(s)\setminus\{*\}$, with unknowns $h(s)_a$ for every $s\in \cT$ and $a\in X(s)\setminus\{*\}$, and it is enough to show that there is a solution in $A[p]$ to this system of equations.

The group $A[p]$ is clearly bounded, so by \cite[Theorem 27.5]{Fu1} it is algebraically compact. Thus, using \cite[Theorem 38.1]{Fu1}, it is enough to show that there is a solution in $A[p]$ to every finite subsystem of the above system of equations. So suppose $s_i<t_i$ in $\cT$ and $a_i\in X(s_i)\setminus\{*\}$ for every $i=1,\dots,n$. Since $\cT$ is directed, we can find $r\in\cT$ such that $r\geq t_i$ for every $i=1,\dots,n$. It is enough to show that there exists
$$h\in\prod_{s\in\cT}(X(s)\wedge A)[p]$$
such that for every $s<t\leq r$ in $\cT$ we have
$$h(s)=\phi_{t,s}(h(t))+\phi_{t,s}(g(t))-g(s).$$

If $s\leq r$ we define
$$h(s):=\phi_{r,s}(g(r))-g(s)\in X(s)\wedge A.$$
We have
$$ph(s)=p\phi_{r,s}(g(r))-pg(s)=\phi_{r,s}(pg(r))-pg(s)= \phi_{r,s}(f(r))-f(s)=0,$$
so $h(s)\in (X(s)\wedge A)[p]$.
For $s\nleq r$ we define $h(s):=0\in (X(s)\wedge A)[p]$, so we have defined
$$h\in\prod_{s\in\cT}(X(s)\wedge A)[p].$$
Now, let $s<t\leq r$ in $\cT$. Then we have
$$\phi_{t,s}(h(t))+\phi_{t,s}(g(t))-g(s)= \phi_{t,s}(\phi_{r,t}(g(r))-g(t))+\phi_{t,s}(g(t))-g(s)= $$
$$\phi_{t,s}(\phi_{r,t}(g(r)))-g(s)=\phi_{r,s}(g(r))-g(s)= h(s),$$
which finishes our proof.
  \item Follows easily from (1).
  \item The result for $H$ follows easily from the result for $G$, since the image of a cotorsion group is cotorsion (see, for example, \cite[page 233]{Fu1}). Thus, it is enough to show that $G(X,A)$ is cotorsion. Let us decompose $A$ into its divisible and reduced parts $A=D\oplus R$. Then, by Lemma \ref{l:oplus}, we have
  $$G(X,A)\cong G(X,D)\oplus G(X,R).$$
  By Theorem \ref{t:p divisible}, we have that $G(X,D)$ is divisible and, in particular, cotorsion. Since the product of cotorsion groups is cotorsion (\cite[page 233]{Fu1}), we are left to show that $G(X,R)$ is cotorsion. For every $s\in\cT$ the set $X(s)$ is finite so the group
  $$X(s)\wedge R=R^{X(s)\setminus\{*\}}$$
  is a product of reduced cotorsion groups and thus reduced cotorsion. By \cite[page 233]{Fu1} the limit
  $$G(X,R)={\lim}_\cT (X\wedge R)$$
  is also reduced cotorsion.
\end{enumerate}
\end{proof}

The following example shows that even for diagrams of finite sets the cokernel $H$ can easily be non trivial. It will also be useful for us for constructing a counterexample in Theorem \ref{t:finite counter}.
\begin{example}\label{e:finite}
  Let $N$ be an infinite set and let $\cT$ be the poset of all finite nonempty subsets of $N$, ordered by inclusion. Clearly $\cT$ is directed and $|\cT|=|N|$. We define a diagram $X:\cT\to\Set^f_*$, by letting $X(s):=s_+$, and if $s\subseteq t$ then the induced map $X(t)\to X(s)$ is the identity on $X(s)$ and $*$ on the rest. It is not hard to verify that in this case we have natural isomorphisms:
  $$G(X,A)\cong \prod_{\alpha\in N}A,$$
  $$K(X,A)\cong \bigoplus_{\alpha\in N}A,$$
  $$H(X,A)\cong \prod_{\alpha\in N}A/\bigoplus_{\alpha\in N}A.$$
\end{example}

\section{A countable directed poset}
In this section we assume that our diagram directed poset $\cT$ is countable. In this case it is not hard to show that there exists a cofinal functor $\NN\lrar\cT$, where $\NN$ is the poset of natural numbers. Thus we can assume without loss of generality that $\cT=\NN$.

So, throughout this section, we let $A$ be an abelian group and $X:\NN\lrar \Set_*$ a tower of pointed sets. For every $n\in\NN$ we let $\phi_n:=X(n\to n-1):X(n)\to X(n-1)$ be the structure map. Recall that we have the functors:
$$G(X,-)={\lim}_\NN (X\wedge (-)):\AbC\to\AbC,$$
$$K(X,-)=({\lim}_\NN X)\wedge (-):\AbC\to\AbC,$$
$$H(X,-)=G(X,-)/K(X,-):\AbC\to\AbC.$$

\begin{prop}\label{p:exact}
  The functors $G(X,-)$ and $H(X,-)$ are left exact. If the diagram $X$ satisfies the Mittag-Leffler condition then the functors $G(X,-)$ and $H(X,-)$ are exact.
\end{prop}

\begin{proof}
  It is easy to see that $G(X,-)$ and $H(X,-)$ are additive. Let $0\to A\to B\to C\to 0$ be an exact sequence in $\AbC$.

  Let $t\in\cT$. The $\ZZ$-module $\bigoplus_{a\in X(t)\setminus\{*\}}\ZZ$ is flat, so the following sequence is also exact
  $$0\to A\otimes (\bigoplus_{a\in X(t)\setminus\{*\}}\ZZ)\to B\otimes(\bigoplus_{a\in X(t)\setminus\{*\}}\ZZ)\to C\otimes(\bigoplus_{a\in X(t)\setminus\{*\}}\ZZ)\to 0.$$
  This sequence is exactly
  $$0\to X(t)\wedge A\to X(t)\wedge B\to X(t)\wedge C\to 0.$$
  Combining these exact sequences we obtain an exact sequence in $\AbC^\cT$
  $$0\to X\wedge A\to X\wedge B\to X\wedge C\to 0.$$
  It follows (see, for instance, {\cite[Lemma VI.2.12]{GJ}}) that we have an exact sequence in $\AbC$
  $$0\to {\lim}_\cT (X\wedge A)\to {\lim}_\cT (X\wedge B)\to {\lim}_\cT (X\wedge C)\to {\lim}^1_\cT (X\wedge A).$$
  This shows that $G(X,-)$ is left exact.

  Applying the same argument as above for $\lim_\cT X$ instead of $X(t)$, we obtain an exact sequence
  $$0\to ({\lim}_\cT X)\wedge A\to ({\lim}_\cT X)\wedge B\to ({\lim}_\cT X)\wedge C\to 0.$$
  We thus obtain the following commutative diagram in $\AbC$
  $$\xymatrix{0\ar[r] & ({\lim}_\cT X)\wedge A\ar[r]\ar@{^{(}->}[d] & ({\lim}_\cT X)\wedge B \ar[r]\ar@{^{(}->}[d] & ({\lim}_\cT X)\wedge C\ar[r]\ar@{^{(}->}[d] & 0\\
  0\ar[r] & {\lim}_\cT (X\wedge A)\ar[r] & {\lim}_\cT (X\wedge B)\ar[r] & {\lim}_\cT (X\wedge C) & }$$
  where every row is exact.
  By the snake lemma we have an exact sequence
  $$0\to H(X,A)\to H(X,B)\to H(X,C),$$
which shows that $H(X,-)$ is left exact.

  In the case that $X$ satisfies the Mittag-Leffler condition we have $\lim^1_\cT (X\wedge A)=0$, so we have an exact sequence
  $$0\to {\lim}_\cT (X\wedge A)\to {\lim}_\cT (X\wedge B)\to {\lim}_\cT (X\wedge C)\to 0.$$
  By the snake lemma, we also have an exact sequence
  $$0\to H(X,A)\to H(X,B)\to H(X,C)\to 0.$$
\end{proof}

\begin{cor}\label{c:exact}
If $B\subseteq A$ is a subgroup then we have a natural inclusion
$$H(X,B)\hookrightarrow H(X,A)$$
and if $X$ satisfies the Mittag-Leffler condition we also have a natural isomorphism
$$H(X,A/B)\cong H(X,A)/H(X,B).$$
  For every $n\geq 1$ the natural inclusion $H(X,A[n])\hookrightarrow H(X,A)$ gives an isomorphism
  $$H(X,A[n])\cong H(X,A)[n]$$
  and if $X$ satisfies the Mittag-Leffler condition the natural inclusion $H(X,nA)\hookrightarrow H(X,A)$ gives an isomorphism
  $$H(X,nA)\cong nH(X,A).$$
  The same results hold for $G$ instead of $H$.
\end{cor}

\begin{proof}
We have an exact sequence
$$0\lrar B\lrar A\lrar A/B\lrar 0.$$
By Proposition \ref{p:exact} the following sequence is exact:
$$0\lrar H(X,B)\lrar H(X,A)\lrar H(X,A/B).$$
This gives the desired natural inclusion.
If $X$ satisfies the Mittag-Leffler condition, Proposition \ref{p:exact} tells us that the following sequence is also exact:
$$0\lrar H(X,B)\lrar H(X,A)\lrar H(X,A/B)\lrar 0,$$
which gives the first isomorphism.
The rest follows in a similar way starting from the exact sequence
$$0\lrar A[n]\lrar A\xrightarrow{n\cdot(-)}nA\lrar 0.$$
\end{proof}

\begin{cor}\label{c:almost divisible}
Suppose $X$ satisfies the Mittag-Leffler condition. Then we have the following:
\begin{enumerate}
  \item If $A$ is $p$-divisible then $G(X,A)$ and $H(X,A)$ are $p$-divisible.
  \item If $A$ is divisible then $G(X,A)$ and $H(X,A)$ are divisible.
  \item If $A$ is almost divisible then $G(X,A)$ and $H(X,A)$ are almost divisible.
\end{enumerate}
\end{cor}

\begin{proof}
\begin{enumerate}
  \item We have $A=pA$ and since $X$ satisfies the Mittag-Leffler condition Corollary \ref{c:exact} gives
  $$H(X,A)= H(X,p A)= p H(X,A).$$
  \item Follows easily from (1).
  \item Using Corollary \ref{c:AD is AC}, it follows from (2), Theorem \ref{t:p divisible gen} (1) and Lemma \ref{l:oplus}.
\end{enumerate}
\end{proof}

\begin{cor}\label{c:mdiv}
  Let $m\geq 2$ and $h\in {\lim}_\NN (X\wedge A)$. Suppose that $X$ satisfies the Mittag-Leffler condition and that for every $n\in\NN$ and $a\in \supp(h(t))$ we have that $m|h(n)_a$. Then $m|[h]$ in $H$.
\end{cor}

\begin{proof}
  It is easy to see that $h\in\lim_\NN(X\wedge(mA))$, so using Corollary \ref{c:exact} we obtain
  $$[h]\in{\lim}_\NN(X\wedge(mA))/({\lim}_\NN X)\wedge (mA)=H(X,mA)\cong mH(X,A).$$
\end{proof}

Our goal now is to obtain Theorem \ref{t:alg compact}. For this we need to first prove a stronger version of Corollary \ref{c:mdiv}, namely:

\begin{prop}\label{p:devide}
  Let $m\geq 2$, $h\in {\lim}_\NN (X\wedge A)$ and $x_1,\dots,x_l\in {\lim}_\NN X$. Suppose that for every $n\in\NN$ and every $a\in \supp(h(n))\setminus\{x_1(n),\dots,x_l(n)\}$ we have that $m|h(n)_a$. Then $m|[h]$ in $H$.
\end{prop}

\begin{proof}
 We define $g(n)\in X(n)\wedge A$ for every $n\in \NN$, recursively.

 For $n=0$ let $a\in X(0)$. If
 $$a\in\supp(h(0))\setminus \{x_1(0),\dots,x_l(0)\}$$
we choose $g(0)_a\in A$ such that $h(0)_a=mg(0)_a$, otherwise we define $g(0)_a:=0$.

Let $n\in\NN$ and suppose we have defined $g(i)\in X(i)\wedge A$ for every $i\leq n$ such that
\begin{enumerate}
  \item $\phi_{n}(g(n))=g(n-1)$,
  \item $\supp(mg(n)-h(n))\subseteq \{x_1(n),\dots,x_l(n)\}$,
  \item $\supp(g(n))\subseteq \supp(h(n))\cup \{x_1(n),\dots,x_l(n)\}$.
\end{enumerate}

We now define $g(n+1)\in X(n+1)\wedge A$.
  Let $s\in X(n)$ such that
  $$W_s:=(\supp(h(n+1))\cup \{x_1(n+1),\dots,x_l(n+1)\})\cap\phi_{n+1}^{-1}(s)\neq\phi.$$
  If
  $$\{x_1(n+1),\dots,x_l(n+1)\}\cap\phi_{n+1}^{-1}(s)\neq\phi$$
  we choose
  $$a_0\in\{x_1(n+1),\dots,x_l(n+1)\}\cap\phi_{n+1}^{-1}(s)\subseteq W_s,$$
  otherwise, we choose any $a_0\in W_s$.

  For
  $$a\in (W_s\setminus\{a_0\})\setminus \{x_1(n+1),\dots,x_l(n+1)\}$$
we choose $g(n+1)_a\in A$ such that $h(n+1)_a=mg(n+1)_a$, and for
$$a\in (W_s\setminus\{a_0\})\cap \{x_1(n+1),\dots,x_l(n+1)\}$$
we define $g(n+1)_a:=0$.
We now define
$$g(n+1)_{a_0}:=g(n)_s-\sum_{a\in W_s\setminus\{a_0\}}g(n+1)_a.$$

This defines $g(n+1)_a$ for every $a\in W_s$. Ranging over all possible $s$ we have thus defined $g(n+1)_a$ for every
$$a\in\supp(h(n+1))\cup \{x_1(n+1),\dots,x_l(n+1)\}.$$
For the remaining $a\in X(n+1)$ we define $g(n+1)_a=0$.
This defines $g(n+1)\in X(n+1)\wedge A$ and it is clear from the definition that
$$\supp(g(n+1))\subseteq \supp(h(n+1))\cup \{x_1(n+1),\dots,x_l(n+1)\}.$$
Using (3) in the induction hypothesis it is also easy to see that $\phi_{n+1}(g(n+1))=g(n).$

Now let $a\in \supp(mg(n+1)-h(n+1))$, so that $mg(n+1)_a\neq h(n+1)_a$, and assume that $a\notin\{x_1(n+1),\dots,x_l(n+1)\}$ (to derive a contradiction).

Clearly
$$a\in\supp(h(n+1))\cup \{x_1(n+1),\dots,x_l(n+1)\},$$
so there exists $s\in X(n)$ such that $a\in W_s$.

If $a\neq a_0$ then $a\in (W_s\setminus\{a_0\})\setminus \{x_1(n+1),\dots,x_l(n+1)\}$ and according to the construction above we have $h(n+1)_a=mg(n+1)_a$, which is a contradiction.

So suppose $a=a_0$. Then
$$a_0\notin \{x_1(n+1),\dots,x_l(n+1)\}$$ and thus
\begin{equation}\label{e:phi}
\{x_1(n+1),\dots,x_l(n+1)\}\cap\phi_{n+1}^{-1}(s)=\phi.
\end{equation}
According to the construction above for every $r\in W_s\setminus\{a_0\}$ we have $h(n+1)_r=mg(n+1)_r$.

If $s=x_i(n)$ for some $1\leq i \leq l$ then $\phi_{n+1}(x_i(n+1))=x_i(n)=s$ and we obtain $x_i(n+1)\in \phi^{-1}_{n+1}(s)$ contradicting (\ref{e:phi}) above. Thus,
$s\notin \{x_1(n),\dots,x_l(n)\}$ and from (2) in the induction hypothesis it follows that $h(n)_s=mg(n)_s$. We thus obtain
$$mg(n+1)_{a_0}=mg(n)_s-\sum_{r\in W_s\setminus\{a_0\}}mg(n+1)_r=h(n)_s-\sum_{r\in W_s\setminus\{a_0\}}h(n+1)_r=h(n+1)_{a_0},$$
which is again a contradiction.

Thus $a\in\{x_1(n+1),\dots,x_l(n+1)\}$ and we obtain
$$\supp(mg(n+1)-h(n+1))\subseteq \{x_1(n+1),\dots,x_l(n+1)\},$$
which finishes the recursive definition.

  Since $\phi_{n+1}(g(n+1))=g(n)$ for every $n\in\NN$ we have that $$g:=(g(n))_{n\in\NN}\in {\lim}_\NN (X\wedge A).$$
  Now $mg-h\in {\lim}_\NN (X\wedge A)$ and for every $n\in\NN$ we have
  $$|mg(n)-h(n)|=|\supp(mg(n)-h(n))|\leq l.$$
  It follows from Proposition \ref{p:bounded} that $mg-h\in ({\lim}_\NN X)\wedge A$. Thus $m[g]-[h]=[mg-h]=0$ and $m[g]=[h]$.
\end{proof}

\begin{thm}\label{t:alg compact}
  If $X$ satisfies the Mittag-Leffler condition then the group $H$ is algebraically compact.
\end{thm}

\begin{proof}
For every $n\in \NN$ we define
$$X'(n):=\bigcap_{s\geq n}\im (X(s)\to X(n))\subseteq X(n).$$
Clearly, by restriction of the structure maps, we can lift $X'$ into a diagram $X':\NN\lrar \Set_*$.
It is not hard to see that we have a natural isomorphism ${\lim}_\NN X'\cong\lim_\NN X$ and, since $X$ satisfies the Mittag-Leffler condition, all the structure maps of $X'$ are surjective.

We also define $(X\wedge A)':\NN\lrar\AbC$ in a similar manner.
Since for every structure map we have $\phi_{n,m}(X(n)\wedge A)=\phi_{n,m}(X(n))\wedge A$, it is not hard to see that
$(X\wedge A)'=X'\wedge A.$

Let $n\in\NN$. Since all the structure maps in $X'$ are surjective, it follows that
${\lim}_{\NN}X\cong {\lim}_{\NN}X'\to X'(n)$ is surjective so
  $({\lim}_{\NN}X)\wedge A\to X'(n)\wedge A$
  is also surjective.
  Thus we have
 $$\psi_n({({\lim}_\NN X)\wedge A})=X'(n)\wedge A.$$
  Since all the structure maps of $X'\wedge A$ are surjective, it follows that the map
  $${\lim}_\NN (X\wedge A)\cong{\lim}_\NN (X\wedge A)'\cong{\lim}_\NN (X'\wedge A)\to X'(n)\wedge A$$
  is surjective and
   $$\psi_n({{\lim}_\NN (X\wedge A)})=X'(n)\wedge A=\psi_n({({\lim}_\NN X)\wedge A}).$$
It follows from \cite[Theorem 0.0.4]{BaSh} that $H$ is cotorsion. Thus, by  \cite[Proposition 54.2]{Fu1}, it remains to show that $u(H)\leq 1$. By \cite[Proposition 3.0.7]{BaSh}, this is equivalent to showing that $l_p(H)\leq\omega$ for every $p\in\PP.$

   Let $p\in\PP.$ We need to show that
   $$p^\omega H\subseteq p^{\omega+1}H.$$
   Let $f\in{\lim}_\NN(X\wedge A)$ such that
   $$[f]\in p^\omega H=\bigcap_{n=0}^\infty p^n H\subseteq H.$$
   It follows from \cite[Lemma 2.0.3]{BaSh} that for every $g\in{\lim}_\NN(X\wedge A)$ and every $n\in\NN$ there exists $g'\in{\lim}_\NN(X\wedge A)$ such that $[g]=[g']$ and $g'(n)=0$. We can thus assume $f(0)=0$.

   Let $n\geq 1$. Then $[f]\in p^n H$ so there exists $g_n\in{\lim}_\NN(X\wedge A)$ such that in $H$ we have
   $[f]=p^n[g_n]$. Thus, in ${\lim}_\NN(X\wedge A)$ we have $f-p^n g_n\in ({\lim}_\NN X)\wedge A$. Again, using \cite[Lemma 2.0.3]{BaSh}, we can assume that for every $n\geq 1$ we have $g_n(n)=0$.

Since for every $k\geq 1$ we have $f-p^k g_k\in({\lim}_\NN X)\wedge A$, we can choose recursively a strictly increasing sequence of positive integers $(l_k)_{k=1}^\infty$ and $x_i\in\lim_\NN X$, $b_{i,k}\in A$ for every $k\geq 1$ and $1\leq i\leq l_k$ such that for every $k\geq 1$ we have
$$f-p^k g_k=\sum_{i=1}^{l_k}x_i b_{i,k}.$$
We define $l_0:=0$ and $g_0=f$. Note that the last formula remains valid also for $k=0$ and we have $g_0(0)=f(0)=0$. For every $n\in\NN$ we define $$Z_n:=\{x_1(n),x_2(n),\dots\}\subseteq X(n).$$
Since $g_n(n)=0$ we have
$$\supp(f(n))\subseteq \{x_1(n),\dots,x_{l_n}(n)\}\subseteq Z_n.$$

We now define $h(n)\in X(n)\wedge A$ for every $n\in \NN$, recursively.

For $n=0$ we define $h(0):=0\in X(0)\wedge A$.

Let $n\geq 1$ and suppose we have defined $h(i)\in X(i)\wedge A$ for every $i<n$ such that
\begin{enumerate}
  \item $\phi_{{n-1}}(h({n-1}))=h({n-2})$,
  \item $\supp(h({n-1}))\subseteq Z_{n-1}$,
  \item For every $m\geq 1$ and every $a\in X(n-1)\setminus \{x_1({n-1}),\dots,x_{l_m}({n-1})\}$ we have $p^{m-1}|h(n-1)_a$ and $ph(n-1)_a=f(n-1)_a$.
\end{enumerate}

We define an equivalence relation on $Z_n$ by letting $a_1\sim a_2$ iff $\phi_n(a_1)=\phi_n(a_2)$. Let $W\subseteq Z_n$ be an equivalence class. We define
$$k_W:=\min\{k\in\NN|W\cap\{x_1(n),\dots,x_{l_k}(n)\}\neq\phi\}\geq 1,$$
and choose $a_0\in W\cap\{x_1(n),\dots,x_{l_{k_W}}(n)\}$.
Let $a\in W\setminus\{a_0\}$. We define
$$k=k_a:=\min\{k\in\NN|a\in\{x_1(n),\dots,x_{l_k}(n)\}\}\geq k_W.$$
We have
$$f-p^{k-1} g_{k-1}=\sum_{i=1}^{l_{k-1}}x_i b_{i,{k-1}},$$
so in particular
$$f(n)_a-p^{k-1} g_{k-1}(n)_a=\sum_{i=1}^{l_{k-1}}(x_i(n) b_{i,{k-1}})_a.$$
Since $a\notin\{x_1(n),\dots,x_{l_{k-1}}(n)\}$ we have $f(n)_a=p^{k-1} g_{k-1}(n)_a$ and $p^{k-1}|f(n)_a$.
If $f(n)_a=0$ or $k=1$ we define $h(n)_a=0$. Otherwise, we define $h(n)_a:=p^{n-2}g_{k-1}(n)_a$, so $p^{k-2}|h(n)_a$ and $ph(n)_a=f(n)_a$.
Note that $h(n)_a\neq 0$ for only a finite number of $a\in W\setminus\{a_0\}$ so we can define
$$h(n)_{a_0}:=h(n-1)_{\phi_n(a_0)}-\sum_{a\in W\setminus\{a_0\}}h(n)_a.$$
We have thus defined $h(n)_a$ for every $a\in W$. Ranging over all equivalence classes this defines $h(n)_a$ for every
$a\in Z_n$, and for the remaining $a\in X(n)$ we define $h(n)_a=0$.

We claim that for almost every equivalence class $W\subseteq Z_n$ we have $\forall a\in W.h(n)_a=0$.
Since $h(n-1)\in X(n-1)\wedge A$ we have that $h(n-1)_a=0$ for almost every $a\in Z_{n-1}$. Thus, it is enough to consider equivalence classes $W\subseteq Z_n$ for which $\forall a\in W.h(n-1)_{\phi_n(a)}=0$. For such equivalence classes it is easy to see that if $\supp(f(n))\cap W=\phi$ then $\forall a\in W.h(n)_a=0$.
We have thus shown that $h(n)\in X(n)\wedge A$.

It is clear from the definition that
$\supp(h({n}))\subseteq Z_{n}$ and $\phi_{{n}}(h({n}))=h({n-1})$ (note that $\phi_n:Z_n\to Z_{n-1}$ is surjective).

Now let $m\geq 1$ and $a\in X(n)\setminus \{x_1({n}),\dots,x_{l_m}({n})\}$. If $a\notin Z_{n}$ then $h(n)_a=f(n)_a=0$ and the result is clear, so suppose $a\in Z_{n}$. Let $W\subseteq Z_n$ be an equivalence class such that $a\in W$.
Suppose $a\neq a_0$. Then $$k=\min\{k\in\NN|a\in\{x_1(n),\dots,x_{l_k}(n)\}\}\geq m+1\geq 2.$$
According to the construction above we have $p^{m-1}|h(n)_a$ and $ph(n)_a=f(n)_a$ (whether $f(n)_a=0$ or not).

Now suppose $a=a_0$. According to the construction above we have $a_0\in W\cap \{x_1(n),\dots,x_{l_{k_W}}(n)\}$. But $a_0\notin\{x_1({n}),\dots,x_{l_m}({n})\}$ so $k_W\geq m+1$.
For every $b\in W\setminus\{a_0\}$ we have $m+1\leq k_W\leq k_b$, so according to the construction above we have $p^{m-1}|h(n)_b$ and $ph(n)_b=f(n)_b$ (whether $f(n)_b=0$ or not).
By definition of $k_W$ we have
$W\cap\{x_1(n),\dots,x_{l_{k_W-1}}(n)\}=\phi$ so
$$\phi_n(a_0)\in X({n-1})\setminus \{x_1({n-1}),\dots,x_{l_{k_W-1}}({n-1})\}.$$
By the induction hypothesis we have
$p^{k_W-2}|h(n-1)_{\phi_n(a_0)}$ and $ph(n-1)_{\phi_n(a_0)}=f(n-1)_{\phi_n(a_0)}$.
But $m+1\leq k_W$ so $p^{m-1}|h(n-1)_{\phi_n(a_0)}$. According to the construction above we have
$$h(n)_{a_0}=h(n-1)_{\phi_n(a_0)}-\sum_{b\in W\setminus\{a_0\}}h(n)_b,$$
so $p^{m-1}|h(n)_{a_0}$ and $ph(n)_{a_0}=f(n)_{a_0}$, which finishes the recursive definition.

Since $\phi_{n+1}(h(n+1))=h(n)$ for every $n\in\NN$ we have that $$h:=(h(n))_{n\in\NN}\in {\lim}_\NN (X\wedge A).$$
Taking $m=1$ we obtain that for every $n\in\NN$ we have
  $$\supp(ph(n)-f(n))\subseteq\{x_1(n),\dots,x_{l_1}(n)\}$$ so
  $|ph(n)-f(n)|\leq l_1.$
  It follows from Proposition \ref{p:bounded} that $ph-f\in ({\lim}_\NN X)\wedge A$. Thus $p[h]-[f]=[ph-f]=0$ and $p[h]=[f]$.
According to Proposition \ref{p:devide}, for every $m\geq 1$ we have $p^{m-1}|[h]$ in $H$. Thus $[h]\in p^\omega H$ and $[f]=p[h]\in p^{\omega+1} H$.
\end{proof}

If $\cT$ is any directed poset, recall that a sub-poset $\cS\subseteq \cT$ is called \emph{cofinal} if the inclusion functor $\cS\hookrightarrow \cT$ is cofinal, or in other words if for every $t\in\cT$ there exists $s\in \cS$ such that $s\geq t$. The cofinality of $\cT$ is the cardinal
$$\cf(\cT):=\min\{|\cS|\::\:\cS\subseteq \cT\text{ is cofinal}\}.$$

\begin{cor}\label{c:alg compact}
If $\cT=\lambda$ is any ordinal and $X$ satisfies the Mittag-Leffler condition then the group $H$ is algebraically compact.
\end{cor}

\begin{proof}
  If $\cf(\lambda)\leq \aleph_0$ then we have a cofinal functor $\NN\to\lambda$ and the result follows from Theorem \ref{t:alg compact}. If $\cf(\lambda)\geq \aleph_1$ then it is easy to see that $\lambda$ is $\aleph_1$-directed so the result follows from Corollary \ref{c:a1 directed}.
\end{proof}

\section{Counterexamples for a directed poset}
Throughout this section, we let $A$ be an abelian group, $\cT$ a directed poset and $X:\cT\lrar \Set_*$ a diagram of pointed sets.

\begin{prop}\label{p:not cotorsion}
Suppose that $A$ is not almost divisible and for every sequence $(d_n)_{n\in\NN}$ of elements in $A\setminus\{0\}$ there exists a sequence $(g_n)_{n\in\NN}$ of elements in ${\lim}_\cT (X\wedge A)$ such that for every sequence of natural numbers $(k_n)_{n\in\NN}$ there exists $s\in\cT$ such that $|g_n(s)|=k_n$ and for every $a\in X(s)$ we have $g_n(s)_a\in\{0,d_n\}$. Then $H(X,A)$ is not cotorsion.
\end{prop}

\begin{proof}
Since $A$ is not almost divisible, we have by Proposition \ref{p:criterion} that there exists a sequence $(r_n)_{n\in\NN}$ in $\ZZ$ such that for every $m\in\NN$ there exists $n>m$ with
$$r_{<m} A\neq r_{<n} A.$$
Clearly we can construct $0=i_0<i_1<\dots$ such that if we let $q_n:=r_{i_n}\cdots r_{i_{n+1}-1}$ we get that for every $m\in\NN$,
$q_{<m+1} A\neq q_{<m} A.$
We thus have
$$q_{< m+1} A\subsetneq q_{< m} A,$$
so there exists $d_m\in A\setminus\{0\}$ such that
$$q_{< m}d_m\notin q_{< m+1}A.$$
Let $(g_n)_{n\in\NN}$ be a sequence of elements in ${\lim}_\cT (X\wedge A)$ defined from $(d_n)_{n\in\NN}$ as in the conditions of the proposition. By \cite[Proposition 2.0.2]{BaSh}, it is enough to show that the system of equations over $H(X,A)$ given by
$$x_n-q_n x_{n+1}=[g_n]\:\:\:(n\in\NN),$$
has no solution in $H(X,A)$.

  So assume (to derive a contradiction) that $([h_n])_{n\in\NN}$ is a solution in $H(X,A)$ to the above system of equations. That is, for every $n\in\NN$ we have $h_n\in{\lim}_\cT (X\wedge A)$ and $[h_n]-q_n[h_{n+1}]=[g_n].$
  Thus
  $$f_n:=g_n-h_n+q_nh_{n+1}\in ({\lim}_\cT X)\wedge A,$$
  and using Proposition \ref{p:bounded} we can define
  $$k_n^0:=\sup\{|f_n(s)|\::\:s\in\cT\}\in\NN.$$
  We now define recursively a sequence of natural numbers $(k_n)_{n\in\NN}$ by the formula
  $$k_n:=(n+1)+\sum_{l=0}^{n}k_l^0+\sum_{l=0}^{n-1}k_l$$
  (so that $k_0=1+k_0^0$). Let $s$ be an element in $\cT$ defined from $(k_n)_{n\in\NN}$ as in the conditions of the proposition.

\begin{lem}
  For every $n\in\NN$ we have
  $$|q_{<n}h_n(s)|\leq |h_0(s)|+\sum_{l=0}^{n-1}|q_{<l}g_l(s)|+ \sum_{l=0}^{n-1}|q_{<l}f_l(s)|$$
\end{lem}

\begin{proof}
We prove the lemma by induction on $n$. When $n=0$ the lemma is clear. Now suppose the lemma is true for some $n\in\NN$ and let us prove it for $n+1$.

In $X(s)\wedge A$ we have
$$f_n(s)=g_n(s)-h_n(s)+q_nh_{n+1}(s).$$
Multiplying by $q_{<n}$ we obtain
$$q_{<n}f_n(s)=q_{<n}g_n(s)- q_{<n}h_n(s)+ q_{<n+1}h_{n+1}(s).$$
It follows that
$$\supp(q_{<n+1}h_{n+1}(s))\subseteq \supp(q_{<n}h_n(s))\cup \supp(q_{<n}g_n(s))\cup \supp(q_{<n}f_n(s)),$$
so
$$|q_{<n+1}h_{n+1}(s)|\leq|q_{<n}h_n(s)|+|q_{<n}g_n(s)|+|q_{<n}f_n(s)|.$$
Using the induction hypothesis we obtain
$$|q_{<n+1}h_{n+1}(s)|\leq|h_0(s)|+\sum_{l=0}^{n-1}|q_{<l}g_l(s)|+\sum_{l=0}^{n-1}|q_{<l}f_l(s)|+|q_{<n}g_n(s)|+|q_{<n}f_n(s)|=$$
$$|h_0(s)|+\sum_{l=0}^{n}|q_{<l}g_l(s)|+\sum_{l=0}^{n}|q_{<l}f_l(s)|,$$
which proves our lemma.
\end{proof}
Now let $n\in\NN$. It follows from the lemma above that
  $$|q_{<n}h_n(s)|\leq |h_0(s)|+\sum_{l=0}^{n-1}|g_l(s)|+ \sum_{l=0}^{n-1}|f_l(s)|\leq |h_0(s)|+\sum_{l=0}^{n-1}k_l+ \sum_{l=0}^{n-1}k_l^0.$$
In $X(s)\wedge(A/q_{<n+1}A)$ we have
$$[q_{<n}f_n(s)]=[q_{<n}g_n(s)]- [q_{<n}h_n(s)]+ [q_{<n+1}h_{n+1}(s)]=[q_{<n}g_n(s)]- [q_{<n}h_n(s)],$$
so
$$|[q_{<n}g_n(s)]|\leq |[q_{<n}f_n(s)]|+|[q_{<n}h_n(s)]|.$$
Since $q_{< n}d_n\notin q_{< n+1}A$ and for every $a\in X(s)$ we have $g_n(s)_a\in\{0,d_n\}$, we obtain that
$$|g_n(s)|=|[q_{<n}g_n(s)]|.$$
We thus have
$$(n+1)+\sum_{l=0}^{n}k_l^0+\sum_{l=0}^{n-1}k_l=k_n=|g_n(s)| =|[q_{<n}g_n(s)]|\leq |[q_{<n}f_n(s)]|+|[q_{<n}h_n(s)]|\leq$$
$$|q_{<n}f_n(s)|+|q_{<n}h_n(s)|\leq k_n^0+|h_0(s)|+\sum_{l=0}^{n-1}k_l+ \sum_{l=0}^{n-1}k_l^0= |h_0(s)|+\sum_{l=0}^{n-1}k_l+ \sum_{l=0}^{n}k_l^0.$$
Since $n$ can be arbitrary large, we obtain a contradiction.
\end{proof}

\begin{thm}\label{t:counter}
Suppose $A$ is not almost divisible. Then there exists a directed diagram $X:\cT\lrar\Set_*$, with surjective connecting homomorphisms and $|\cT|=2^{\aleph_0}$, such that $H(X,A)$
is not cotorsion.
\end{thm}

\begin{proof}
Let $\Phi$ be the set of all functions $\NN\to\NN$, and let $\cT$ be the poset of all finite nonempty subsets of $\Phi$, ordered by inclusion. Clearly $\cT$ is directed and $|\cT|=|\Phi|=2^{\aleph_0}$. We define a diagram $X:\cT\to\Set_*$, by letting $X(s):=(s\times \NN)_+$, and if $s\subseteq t$ then the induced map $X(t)\to X(s)$ is the identity on $X(s)$ and $*$ on the rest. Clearly $X$ has surjective connecting homomorphisms.

Let $(d_n)_{n\in\NN}$ be a sequence of elements in $A\setminus\{0\}$. Let $n\in\NN$. For every $s\in\cT$ we define $g_n(s)\in X(s)\wedge A$ by letting $g_n(s)_a:=d_n$ if $a\in\{(k,l)\in s\times \NN\:|\:l<k_n\}$ and $g_n(s):=0$ otherwise. It is not hard to see that $g_n:=(g_n(s))_{s\in\cT}\in{\lim}_\cT X\wedge A.$

Now let $k\in\Phi$ be a sequence of natural numbers. We define $s:=\{k\}\in\cT$. Then clearly $|g_n(s)|=k_n$ and for every $a\in X(s)$ we have $g_n(s)_a\in\{0,d_n\}$. By Proposition \ref{p:not cotorsion}, we have that $H(X,A)$ is not cotorsion.
\end{proof}

Let us summarize what we have shown about the question of when is $H(X,A)$ cotorsion.

\begin{cor}\
\begin{enumerate}
  \item The group $H(X,A)$ is cotorsion if $|\cT|=\aleph_0$ and $X$ satisfies the Mittag-Leffler condition (see Theorem \ref{t:alg compact}).
  \item The group $H(X,A)$ is cotorsion if $A$ is almost uniquely divisible (see Theorem \ref{t:p divisible gen}).
  \item If $A$ is not almost divisible there exists a pro-set $X:\cT\lrar\Set_*$, with surjective connecting homomorphisms and $|\cT|=2^{\aleph_0}$, such that $H(X,A)$ is not cotorsion (see Theorem \ref{t:counter}).
\end{enumerate}
\end{cor}

An abelian group is almost divisible iff it is isomorphic to a direct sum of a divisible group and a bounded group. Thus almost divisible groups are groups of the form
$$(\bigoplus_{\alpha<\lambda_0} \QQ )\oplus\bigoplus_{p\in \PP}\bigoplus_{\alpha<\lambda_p}\ZZ(p^\infty)\oplus B,$$
with $(\lambda_p)_{p\in\PP\cup \{0\}}$ cardinals and $B$ bounded. An abelian group is almost uniquely divisible iff it is isomorphic to a direct sum of a uniquely divisible group and a bounded group. Thus almost uniquely divisible groups are groups of the form
$$(\bigoplus_{\alpha<\lambda_0} \QQ )\oplus B,$$
with $\lambda_0$ a cardinal and $B$ bounded.
Thus, the following question remains:
\begin{ques}
  Let $p$ be a prime number and suppose $A=\ZZ(p^\infty)$. Is $H(X,A)$ cotorsion whenever $X$ satisfies the Mittag-Leffler condition?
\end{ques}

Now suppose that $X:\cT\lrar \Set^f_*$ is a diagram of finite pointed sets.
If the diagram $\cT$ is countable then, by Theorem \ref{t:alg compact}, $H(X,A)$ is cotorsion (even algebraically compact) for \textbf{any} abelian group $A$. (Note that a diagram of finite sets always satisfies the Mittag-Leffler condition.) The following Theorem shows that for uncountable $\cT$, Theorem \ref{t:p divisible} is best possible.
\begin{thm}\label{t:finite counter}
   Suppose $A$ is not cotorsion and $\lambda>\aleph_0$ a cardinal. Then there exists a directed diagram $X:\cT\to\Set^f_*$, with $|\cT|=\lambda$, such that $H(X,A)$ is not cotorsion.
\end{thm}

\begin{proof}
 Let $\cT$ and $X$ be as in Example \ref{e:finite} with $N:=\lambda$. We only need to show that
 $$H(X,A)\cong \prod_{\alpha\in \lambda}A/\bigoplus_{\alpha\in \lambda}A$$
  is not cotorsion.

  Since $A$ is not cotorsion, by \cite[Proposition 2.0.2]{BaSh} we can find a sequence $(a_n)_{n\in\NN}$ in $A$ such that there is no solution in $A$ to the system of equations:
  $$x_n-(n+1)x_{n+1}=a_n,$$
  for $n\in\NN$. For every $n\in\NN$ we define $g_n\in \prod_{\alpha\in \lambda}A$ by $g_n(\alpha):=a_n$. Then $[g_n]\in H$ and by \cite[Proposition 2.0.2]{BaSh} it is enough to show that the system of equations:
  $$x_n-(n+1)x_{n+1}=[g_n],$$
  for $n\in\NN$, has no solution in $H$.

  So suppose (to derive a contradiction) that the above system of equations has a solution $x_n=[h_n]$ in $H$, with $h_n\in \prod_{\alpha\in \lambda}A$. Let $n\in\NN$. Then
  $$[h_n]-(n+1)[h_{n+1}]-[g_n]=[0],$$
  so
  $$h_n-(n+1)h_{n+1}-g_n\in\bigoplus_{\alpha\in\lambda}A.$$
  and thus, the set
  $$u_n:=\{\alpha\in\lambda\:|\:h_n(\alpha)-(n+1)h_{n+1}(\alpha)\neq g_n(\alpha)\}$$
  is finite. Since $\lambda$ is a cardinal greater than $\aleph_0$, we can find
  $$\beta\in\lambda\setminus\bigcup_{n\in\NN}u_n.$$
  It is easily seen that $x_n=h_n(\beta)$ is a solution in $A$ to the system of equations:
  $$x_n-(n+1)x_{n+1}=a_n,$$
  contradicting the fact that there is no such solution.
\end{proof}

The following summarizing corollary gives a complete answer to the question of when is $H(X,A)$ cotorsion for every pointed pro-finite set $X$ of a given cardinality.
\begin{cor}\
\begin{enumerate}
  \item For any abelian group $A$, the following assertion holds:
  $H(X,A)$ is cotorsion for every pointed pro-finite set $X:\cT\to\Set^f_*$ with $|\cT|=\aleph_0$ (see Theorem \ref{t:alg compact}).
  \item If $\lambda>\aleph_0$ is a cardinal then the assertion '$H(X,A)$ is cotorsion for every pointed pro-finite set $X:\cT\to\Set^f_*$ with $|\cT|=\lambda$' holds iff $A$ is cotorsion (see Theorems \ref{t:p divisible} and \ref{t:finite counter}).
\end{enumerate}
\end{cor}



\begin{thebibliography}{ModelPro}
\bibitem[AR]{AR} Adamek J., Rosicky J. {\em Locally Presentable and Accessible Categories}, Cambridge University Press, Cambridge, 1994.

\bibitem[Akh]{Akh}
Akhtiamov D. \emph{Homologies of inverse limits of groups}, preprint, arXiv:1901.01125, 2018.

\bibitem[Bal]{Bal} Balcerzyk S. {\em On algebraically compact groups of I. Kaplansky}, Fund. Math. 44, 1957, p. 91--93.

\bibitem[BaSh]{BaSh} Barnea I., Shelah S. {\em The abelianization of inverse limits of groups}, Israel Journal of Mathematics 227.1, 2018, p. 455--483.



\bibitem[Fu1]{Fu1} Fuchs L. {\em Infinite Abelian Groups, Volume I}, Pure and Applied Mathematics, Vol 36, Academic Press, New York, 1970.


\bibitem[GJ]{GJ} Goerss P. G., Jardine J. F. {\em Simplicial Homotopy Theory}, Progress in Mathematics, Vol. 174, Birkh\"auser, Basel, 1999.


\bibitem[HaVe]{HaVe} Hartl M., Vespa C. {\em Quadratic functors on pointed categories}, Advances in Mathematics 226.5, 2011, p. 3927--4010.

\bibitem[Hul]{Hul} Hulanicki A. {\em The structure of the factor group of an unrestricted sum by the restricted sum of abelian groups}, Bull. Acad. Polon. Sci. 10, 1962, p. 77--80.






\end{thebibliography}
\end{document}